\documentclass[reqno,11pt]{amsart}
\usepackage{amsmath,amssymb}
\usepackage{graphicx}
\usepackage[top=1in,bottom=1in,left=1in,right=1in]{geometry}

\newtheorem{thm}{Theorem}

\newtheorem{prop}{Proposition}[section]

\theoremstyle{definition}
\newtheorem{rem}{Remark}[section]
\numberwithin{equation}{section}

\usepackage{graphicx, color}

\newcommand{\eps}{\varepsilon}

\newcommand{\be}{\begin{equation} \label}
\newcommand{\ee}{\end{equation}}
\newcommand{\bea}{\begin{eqnarray}\label}
\newcommand{\eea}{\end{eqnarray}}
\newcommand{\bas}{\begin{eqnarray*}}
\newcommand{\eas}{\end{eqnarray*}}

\def\R{\mathbb R}

\def\be{\begin{equation}}
\def\ee{\end{equation}}

\title[The heat equation with combined nonlinearities]
{Discontinuous critical Fujita exponents for the heat equation with combined nonlinearities}

\author{Mohamed Jleli, Bessem Samet, Philippe Souplet}

\subjclass[2010]{35K05; 35B44; 35B33}

\keywords{Heat equation; combined nonlinearities; Fujita critical exponent; blow-up; global existence}

\begin{document}

\maketitle

\begin{abstract}
We consider the nonlinear heat equation $u_t-\Delta u =|u|^p+b |\nabla u|^q$ in $(0,\infty)\times \R^n$, where $n\geq 1$, $p>1$, $q\geq 1$ and $b>0$. First, we focus our attention on positive solutions and obtain an optimal Fujita-type result:
any positive solution blows up in finite time
if $p\leq 1+\frac{2}{n}$ or $q\leq 1+\frac{1}{n+1}$,  while global classical positive solutions exist for suitably small  initial data
 when $p>1+\frac{2}{n}$ and $q> 1+\frac{1}{n+1}$.  Although finite time blow-up cannot be produced by the gradient term alone and should be considered as an effect of the source term $|u|^p$,
this result shows that the gradient term induces an interesting phenomenon of discontinuity of the critical Fujita exponent,
jumping from $p=1+\frac{2}{n}$ to $p=\infty$ as $q$ reaches the value $1+\frac{1}{n+1}$ from above. 
Next, we investigate the case of sign-changing solutions and 
show that 
if $p\le  1+\frac{2}{n}$ or $0<(q-1)(np-1)\le 1$,
then the solution blows up in finite time for any nontrivial initial data with nonnegative mean. Finally,
a Fujita-type result, with a different critical exponent, is % also 
obtained for  sign-changing solutions to the inhomogeneous version of this problem.
\end{abstract}

\section{Introduction}

Consider the problem
\begin{eqnarray}\label{pbm1}
\left\{\begin{array}{lllll}
u_t-\Delta u&=&|u|^p+b|\nabla u|^q &\mbox{in}& (0,\infty)\times \mathbb{R}^n,
\vspace{1mm}\\
u(0,x)&=&u_0(x) &\mbox{in}&  \R^n,
\end{array}
\right.
\end{eqnarray}
 and its inhomogeneous version
\begin{eqnarray}\label{pbm1b}
\left\{\begin{array}{lllll}
u_t-\Delta u&=&|u|^p+b|\nabla u|^q+h(x) &\mbox{in}& (0,\infty)\times \mathbb{R}^n,
\vspace{1mm}\\
u(0,x)&=&u_0(x) &\mbox{in}&  \R^n.
\end{array}
\right.
\end{eqnarray}
Here $p>1$, $q\ge 1$, $b>0$ and $0\le h\in BC^1(\R^n)$, where 
$$
BC^1(\R^n)=\bigl\{\phi\in C^1(\R^n):\, \phi, \nabla\phi\in L^\infty(\R^n)\bigr\}.
$$
Recall (see e.g. \cite[Section 35]{QS07}) that problem \eqref{pbm1b} (\eqref{pbm1} is a special case) 
is locally well posed \footnote{Local well-posedness remains true under weaker regularity assumptions on the initial data;
however we shall not discuss this here, since our focus is on the large time behavior of solutions.} for $u_0\in BC^1(\R^n)$.
Namely, there exists $\sigma>0$ such that problem \eqref{pbm1b} admits  a unique local classical solution $u\in C^{1,2}((0,\sigma)\times \R^n)$ with $u,\nabla u\in BC([0,\sigma)\times \R^n)$, where $BC([0,\sigma)\times \R^n)$ denotes the space of bounded and continuous functions in $[0,\sigma)\times \R^n$. Denoting by $T=T(u_0)\in (0,\infty]$ the existence time of the (unique) maximal classical solution,
we have blow-up alternative in $C^1$ norm, i.e.
\be\label{BUalt}
T<\infty\Longrightarrow \lim_{t\to T^-}\|u(t)\|_{C^1}=\infty.
\ee
Moreover, we have $u>0$ in $(0,T)\times \R^n$ whenever $u_0\ge 0$ and $u_0\not\equiv 0$. Actually, it can be shown that the $L^\infty$ norm cannot stay bounded if $T<\infty$ (see Proposition \ref{PR1}). The first aim of this paper is to give optimal Fujita-type results for 
 problems \eqref{pbm1} and \eqref{pbm1b}. We mention below some motivations for studying such problems.

In the case $b=0$, problem \eqref{pbm1} reduces to the semilinear heat equation 
\begin{eqnarray}\label{pfuj}
\left\{\begin{array}{lllll}
u_t-\Delta u&=&|u|^p &\mbox{in}& (0,\infty)\times \mathbb{R}^n,
\vspace{1mm}\\
u(0,x)&=&u_0(x) &\mbox{in}&  \R^n.
\end{array}
\right.
\end{eqnarray}
In \cite{Fuj66}, Fujita proved the following results for problem \eqref{pfuj}:
\begin{itemize}
\item[$\bullet$] If $1<p<1+\frac{2}{n}$, then problem \eqref{pfuj} admits no global positive solutions;
\smallskip
\item[$\bullet$] If $p>1+\frac{2}{n}$, then for positive initial values bounded by a sufficiently small Gaussian, problem \eqref{pfuj} admits positive global solutions. 
\end{itemize}
\smallskip
Later, it was shown that $1+\frac{2}{n}$ belongs to the blow-up case (see \cite{A,H,sugi_ojm75,K,W}). The exponent
\be\label{pf}
p_F=1+\frac{2}{n}
\ee
is said to be critical in the sense of Fujita. Note that in \cite{W}, it was also shown
that if $p>p_F$ and $u_0\in L^{\frac{n(p-1)}{2}}(\mathbb{R}^n)$ is sufficiently small, then we have global positive solutions. 

Consider now the problem
\begin{eqnarray}\label{pbVHJ}
\left\{\begin{array}{lllll}
u_t-\Delta u&=&|\nabla u|^q&\mbox{in}& (0,\infty)\times \mathbb{R}^n,
\vspace{1mm}\\
u(0,x)&=&u_0(x) &\mbox{in}&  \R^n.
\end{array}
\right.
\end{eqnarray}
Problem \eqref{pbVHJ} is the well-known viscous Hamilton-Jacobi 
equation. Such type of equations 
 arise in stochastic control theory (see \cite{Lions}) and have also been used to describe a model for growing random interfaces (see e.g. \cite{KA,KS}). On the other hand, it is known that if $u_0\in BC^1(\R^n)$, 
problem \eqref{pbVHJ} admits a unique local in time solution, classical for $t>0$.    Moreover, $u$ exists globally for all $t\geq 0$, and satisfies
$$
\|u(t)\|_\infty \leq \|u_0\|_\infty\quad\mbox{and}\quad \|\nabla u(t)\|_\infty \leq     \|\nabla u_0\|_\infty,\quad t\geq 0.
$$
For more details, see e.g. \cite{AB,B,BSW,BSW2,BL,GGK}. The large time behavior of positive solutions to problem \eqref{pbVHJ} 
was investigated in \cite{BSW2, LS, BKL04, Gil05}.
In particular, under the assumption $u_0\geq 0$,  $u_0\not\equiv 0$ and $u_0\in BC^1(\R^n)\cap L^1(\mathbb{R}^n)$,
denoting $I_\infty=\lim_{t\to \infty} \|u(t)\|_1$, it was shown in \cite{LS} that 
\begin{itemize}
\item[$\bullet$] If $1\leq q\leq \frac{n+2}{n+1}$, then $I_\infty=\infty$, for all $u_0$;
\smallskip
\item[$\bullet$] If $\frac{n+2}{n+1}<q<2$, then both $I_\infty=\infty$ and $I_\infty<\infty$ occur;
\smallskip
\item[$\bullet$] If $q\geq 2$, then $I_\infty<\infty$, for all $u_0$. 
\end{itemize}
 Moreover, a classification of the pointwise asymptotic behaviors of positive solutions of \eqref{pbVHJ} 
as $t\to\infty$ was obtained in \cite{BKL04, Gil05}.

Next, consider the problem
 \begin{eqnarray}\label{pbminesb}
\left\{
\begin{array}{lllll}
u_t-\Delta u&=&|u|^{p-1} u-b |\nabla u|^q &\mbox{in}& (0,\infty)\times \R^n,
\vspace{1mm}\\
u(0,x)&=&u_0(x) &\mbox{in}&  \R^n, 
\end{array}
\right.
\end{eqnarray}
where $p>1$, $q\geq 1$ and $b>0$.
It is known that problem \eqref{pbminesb} admits a unique, maximal in time, classical solution $u\geq 0$, 
for all $u_0\in BC^1(\R^n)$ with $u_0\geq 0$.
Problem \eqref{pbminesb}  (also in bounded domains with Dirichlet boundary conditions) was first proposed in \cite{CW} 
in order to investigate the possible
effect of a damping gradient term on global existence or nonexistence. 
On the other hand, a model in population dynamics was proposed in \cite{S}, where this problem  describes the evolution of the population density of a biological species, under the effect of certain natural mechanisms. 
  For a detailled review of this problem, we refer to the survey paper \cite{S2}. 
For any $p>p_F$ and $b>0$ (and any $q$), it follows from a comparison argument with 
 the case $b=0$ that there always exist positive global solutions to  problem \eqref{pbminesb}.
Moreover, if $1<q<\frac{2p}{p+1}$ or if $q=\frac{2p}{p+1}$ and $b$ is large, 
 then both blowing-up and stationary positive solutions do exist for any $p>1$. Also some global positive solutions exist when $q\ge p>1$.
Therefore, no Fujita-type phenomenon occurs in these cases. As an answer to an open problem proposed in \cite{S2}, 
it was shown in~\cite{MP02}, by a test function approach, that if $q=\frac{2p}{p+1}$, $1<p<p_F$ and $b$ is small enough, 
then problem \eqref{pbminesb} admits no positive global solutions. Hence, $p_F$ is critical in this case. 
 As for the remaining range $\frac{2p}{p+1}<q<p$ and $p>p_F$, it seems to be unknown whether \eqref{pbminesb} admits any global positive solutions.
 
For the inhomogeneous problem \eqref{pbm1b} without gradient term (i.e. for $b=0$),  it was proved in \cite{BLZ}  (see also \cite{Zhang} for the case of positive solutions) that the Fujita exponent becomes different from the homogeneous case and is given by $p^*:
=\frac{n}{n-2}$ ($n\geq 2$). Namely,  it was shown that 
\begin{itemize}
\item[$\bullet$]  If $p<p^*$ and $\int_{\R^n}h(x)\,dx>0$, then the problem admits no global solutions;
\smallskip
\item[$\bullet$] If $n\geq 3$ and $p>p^*$, then for any $\delta>0$, there exists $\varepsilon>0$ such that the problem  has global solutions provided that $|h(x)|,|u_0(x)|\leq \frac{\varepsilon}{(1+|x|^{n+\delta})}$ regardless of whether or not $\int_{\R^n}h(x)\,dx>0$. 
\end{itemize} 
Moreover,  it was  proved that when $p=\frac{n}{n-2}$ ($n\geq 3$) and under a decay condition on $h(x)$ as $|x|\to \infty$, the problem admits no global solutions.

As far as we know, the study of Fujita-type results for problems  \eqref{pbm1} and \eqref{pbm1b} was not considered previously. Note that a similar question was studied in the case of nonlinear wave equations of the type
$$
u_{tt}-\Delta u = |u|^p+|u_t|^q,\quad t>0,\, x\in \mathbb{R}^n,
$$ 
where $p,q>1$ (see \cite{ZH}).

The rest of the paper is organized as follows. In Section \ref{sec2}, problem 
\eqref{pbm1} is investigated. We first obtain an optimal Fujita-type result for positive solutions.  Next, we discuss the case 
 of possibly sign-changing solutions  under a nonnegative initial mean condition $\int_{\R^n} u_0(x)\,dx\ge 0$, and derive a finite time blow-up result. In Section \ref{sec3}, we investigate the case of possibly sign-changing solutions to the nonhomogeneous problem \eqref{pbm1b} %pronblem
 under the condition $\int_{\R^n}h(x)\,dx>0$
 and obtain a Fujita-type result, with a different critical exponent.

\section{The study of problem \eqref{pbm1}}\label{sec2}

We first focus our attention on positive solutions
and thus consider problem \eqref{pbm1} with $u_0\geq 0$, $u_0\not\equiv 0$. 
Our first main result is given by the following theorem.

\medskip

\begin{thm}\label{T1}
Let $n\geq 1$, $p>1$, $q\ge 1$ and $b>0$. 
\begin{itemize}
\item[(i)] Assume 
$$
p\le p_F\quad\hbox{or}\quad  q\le 1+\frac{1}{n+1},
$$
where $p_F$ is given by \eqref{pf}. Then for any $u_0\in BC^1(\R^n)$ such that $u_0\ge 0$ and $u_0\not\equiv 0$, the solution to problem \eqref{pbm1} blows up in finite time.
\item[(ii)] Assume 
\be\label{hyp2}
p>p_F\quad\hbox{and}\quad  q>1+\frac{1}{n+1}.
\ee
Then problem \eqref{pbm1} admits positive global classical solutions for suitably small initial data. Namely, this is for instance the case if $u_0\in BC^1(\R^n)$ satisfies
$$
0\le u_0(x)\le  \eps e^{-\frac{|x|^2}{4}}\quad\mbox{and}\quad u_0\not\equiv 0, 
$$
with $\eps>0$ small enough.
\end{itemize}
\end{thm}

\begin{rem}
(i) In view of the results recalled in introduction, 
the finite time blow-up cannot be produced by the gradient term alone and should be considered as an effect of the source term $|u|^p$.
However our result reveals the occurence of an interesting phenomenon of discontinuity of the Fujita exponent
 in the presence of the gradient term.
Namely, as long as $q>1+\frac{1}{n+1}$, the Fujita exponent remains the same as for the problem without gradient term,
i.e. $p=1+\frac{2}{n}$, and when $q$ becomes equal to, or smaller than, $1+\frac{1}{n+1}$, the Fujita exponent jumps to $p=\infty$
(i.e. no positive global solutions exist whatever the value of $p>1$).
A similar situation will be encountered in the next section for the inhomogeneous problem. 
\smallskip

(ii) A related, though different, phenomenon was observed in \cite{AE93} for positive solutions of the reaction-convection-diffusion problem 
\begin{eqnarray}\label{CRD}
\left\{
\begin{array}{lllll}
u_t-\Delta u&=&|u|^{p-1} u+a\cdot \nabla(u^q) &\mbox{in}& (0,\infty)\times \R^n, 
\vspace{1mm}\\
u(0,x)&=&u_0(x) &\mbox{in}&  \R^n, 
\end{array}
\right.
\end{eqnarray}
where $p>1$, $q\geq 1$ and $a\ne 0$.
Namely, the Fujita exponent for \eqref{CRD} was shown to be given by:
$$p_1(n,q):=
\begin{cases}
1+{2\over n},&\quad \hbox{ if } q\ge \frac{n+1}{n},\\
1+{2q\over n+1},&\quad \hbox{ if } 1<q<\frac{n+1}{n},\\
1+{2\over n},&\quad \hbox{ if } q=1.
\end{cases}
$$
A first change of critical exponent thus occurs when $q$ decreases from $\frac{n+1}{n}$, but in a continuous way,
and a discontinuous jump of the critical exponent then occurs when $q$ reaches $1$ from above. We refer to \cite{DL00,Lev90} for general survey articles on critical Fujita exponents and to \cite{CDW07,LoQu11,FiKi12} for other non-standard behaviors regarding such exponents (appearing in certain time-nonlocal problems).
\end{rem}

The following proposition  will be used in the  proofs of Theorems \ref{T1} and \ref{T3}.

\begin{prop}\label{PR1}
Consider problem \eqref{pbm1b} with $n\geq 1$, $p>1$, $q\ge 1$, $b\in\R$, $u_0, h\in BC^1(\mathbb{R}^n)$ and $h\ge 0$.
\begin{itemize}
\item[(i)] Assume that 
\be\label{asi}
 u(t,x)\le M,\quad (t,x)\in (0,\min(T,\tau))\times \R^n,
\ee
for some $M,\tau>0$. Then there exists  a positive constant $N=N(p,M,\tau,\|u_0\|_{C^1},\|\nabla h\|_\infty)$ such that 
\be\label{gb}
|\nabla u(t,x)|\le N,\quad (t,x)\in (0,\min(T,\tau))\times \R^n.
\ee
\item[(ii)] If $T<\infty$, then we have
$$
\limsup_{t\to T^-} \,\Bigl(\sup_{x\in\R^n}u(t,x)\Bigr)=\infty.
$$
\end{itemize}
\end{prop}

\begin{proof}
(i)  By the maximum principle, we have 
$$\inf_{t\in (0,T),\,x\in\R^n} u(t,x)\ge m:=\inf_{x\in\R^n}u_0(x)>-\infty,$$
hence
$$|u(t,x)|\le \max(M,|m|),\quad (t,x)\in (0,\min(T,\tau))\times \R^n.$$
If $q\le 2$, the assertion then follows from standard gradient bounds (see e.g. \cite{LSU,Lie86}). Let us thus consider the case $q>2$ (but the following argument actually works for any $q>1$). Let $\sigma\in \{-1,1\}$ and denote
$$
z_i(t,x):=\sigma\partial_{x_i}u(t,x),\quad  (t,x)\in (0,\min(T,\tau))\times \R^n,\,\, i\in\{1,\cdots,n\}.
$$
It can be easily seen that  for all $i\in\{1,\cdots,n\}$, $z_i$ solves
\be \label{zieq}
(z_i)_t-\Delta z_i=p|u|^{p-2}uz_i+bq|\nabla u|^{q-2}\nabla u\cdot\nabla z_i +\sigma\partial_{x_i}h(x), \quad (t,x)\in (0,\min(T,\tau))\times \R^n,
\ee
where $\cdot$ denotes the scalar product in $\mathbb{R}^n$. Hence, using \eqref{asi}, for all $i\in\{1,\cdots,n\}$, there holds
$$
(z_i)_t-\Delta z_i\le K|z_i|+q|b||\nabla u|^{q-1}|\nabla z_i|
+ |\partial_{x_i}h|,\quad  (t,x)\in (0,\min(T,\tau))\times \R^n,
$$
where $K:= p\max(M^{p-1},|m|^{p-1})$. On the other hand, observe that for all $i\in\{1,\cdots,n\}$, 
$$
\overline z_i(t,x):=\bigl(\|z_i(0,\cdot)\|_\infty+K^{-1}\|\partial_{x_i}h\|_\infty\bigr)e^{Kt},\quad (t,x)\in (0,\min(T,\tau))\times \R^n
$$
is a supersolution to \eqref{zieq}. Hence, one deduces that 
$$
\pm\partial_{x_i}u(t,x)\le \bigl(\|z_i(0,\cdot)\|_\infty+K^{-1}\|\partial_{x_i}h\|_\infty\bigr)
e^{K\tau},\quad (t,x)\in (0,\min(T,\tau))\times \R^n,\,\, i\in\{1,\cdots,n\},
$$
which yields the desired gradient bound \eqref{gb}.

\medskip

(ii) This is an immediate consequence of \eqref{BUalt} and assertion (i).
\end{proof}

Now, we are able to prove Theorem \ref{T1}.

\begin{proof}[Proof of Theorem \ref{T1}.]

(i)  The case $p\le 1+\frac{2}{n}$ is immediate.
Indeed it is well known that the solution $v>0$ to the problem
\begin{eqnarray*}
\left\{\begin{array}{lllll}
v_t-\Delta v&=&v^p &\mbox{in}& (0,\infty)\times \mathbb{R}^n,\\
v(0,x)&=&u_0(x) &\mbox{in}& \mathbb{R}^n
\end{array}
\right.
\end{eqnarray*}
blows up in finite time in $L^\infty$ norm (cf.~the introduction).
 Since, by the comparison principle, we have $u\ge v$ as long as both solutions exist, it follows that $T<\infty$.

\smallskip

Let us thus assume $1\le q\le 1+\frac{1}{n+1}$. Consider the corresponding diffusive Hamilton-Jacobi equation
\begin{eqnarray}\label{pbm2}
\left\{\begin{array}{lllll}
w_t-\Delta w&=&b|\nabla w|^q &\mbox{in}& (0,\infty)\times \mathbb{R}^n,\\
w(0,x)&=&u_0(x) &\mbox{in}& \R^n.
\end{array}
\right.
\end{eqnarray}
We know that problem \eqref{pbm2} admits a unique global classical solution. By \cite[Theorems 2 and 3]{Gil05}, since $1\le q\le 1+\frac{1}{n+1}$, there exists $\ell\in (0,\infty)$ such that
\be\label{cvloc}
\lim_{t\to\infty} w(t,x)=\ell, \quad\hbox{ uniformly on compact subsets of $\R^n$.}
\ee
Now, we argue by contradiction by supposing that $u$ exists globally. By the comparison principle (see e.g. \cite[Section 52]{QS07}), it follows that
\be \label{inquw}
u(t,x)\ge w(t,x),\quad (t,x)\in (0,\infty)\times \R^n.
\ee
Let $R>0$ to be chosen below. By \eqref{cvloc} and \eqref{inquw}, there exists $t_0=t_0(R)>0$ such that
\be\label{cvloc2}
u(t_0,x)\ge w(t_0,x)\ge \frac{\ell}{2},\quad\hbox{ for all $x\in B_R$,}
\ee
where we denote
$B_\rho:=\left\{x\in \mathbb{R}^n:\, |x|< \rho\right\}$ for all $\rho>0$.
From this lower bound, we shall reach a contradiction by using the classical Kaplan eigenfunction method \cite{Kap63} 
combined with a rescaling argument. Let $\lambda >0$ be the first eigenvalue of $-\Delta$ in $H_0^1(B_1)$ and let
$\varphi>0$ be the corresponding eigenfunction normalized by $\int_{B_1}\varphi\,dx =1$. Set 
$$
\varphi_R(x):=R^{-n}\varphi(R^{-1}x), \quad x\in B_R.
$$
Multiplying \eqref{pbm1} by $\varphi_R$, integrating by parts over $B_R$, and using  $\Delta \varphi_R=-\lambda R^{-2}\varphi_R$, for all $t>0$, we obtain
$$
\begin{aligned}
\frac{d}{dt} \int_{B_R} u(t)\varphi_R \,dx +\lambda R^{-2} \int_{B_R} u \varphi_R\,dx 
+\int_{\partial B_R} &(u\partial_\nu \varphi_R-\varphi_R\partial_\nu u)\,d\sigma
=\int_{B_R} \left(u^p +b |\nabla u|^q\right)\varphi_R\,dx,
\end{aligned}
$$ 
where $\partial_\nu$ denotes  the outward normal derivative on $\partial B_R$. Since $b>0$, $\varphi_R=0$ and $\partial_\nu \varphi_R\le 0$ on $\partial B_R$, one deduces that 
$$
\frac{d}{dt} \int_{B_R} u(t)\varphi_R \,dx\geq \int_{B_R} u^p(t)\varphi_R\,dx - \lambda R^{-2} \int_{B_R} u(t) \varphi_R\,dx,\quad t>0.
$$
Next, using $\int_{B_R}\varphi_R\,dx=1$ and Jensen's inequality, it follows that $y(t):=\int_{B_R} u(t)\varphi_R \,dx$ satisfies
\be\label{ineqdiff}
y'(t)\ge y^p(t)-\lambda R^{-2}y(t),\quad t>0.
\ee
But, on the other hand, \eqref{cvloc2} implies
\be \label{inqy}
y(t_0)\ge \frac{\ell}{2} \int_{B_R} \varphi_R\, dx=\frac{\ell}{2}.
\ee
Choosing 
$R>\sqrt{\lambda} \left(\frac{\ell}{2}\right)^{\frac{1-p}{2}}$,
from \eqref{inqy}, we obtain
$y^{p-1}(t_0)>\lambda R^{-2}$,
which implies that  the right-hand side of \eqref{ineqdiff} at $t=t_0$ is positive. But the differential inequality \eqref{ineqdiff} then cannot have a global solution for $t>t_0$: a contradiction. We conclude that $T<\infty$.

\smallskip 

(ii) Let
$$
z(t,x)=\eps(t+1)^k \phi(t,x),\quad (t,x)\in (0,\infty)\times \mathbb{R}^N,
$$
with $\eps,k>0$ to be chosen and
$\phi(t,x)=(t+1)^{-\frac{n}{2}}\exp[{-{|x|^2\over 4(t+1)}}]$.
Set $\mathcal{P}z:=z_t-\Delta z-z^p-b|\nabla z|^q$.
Since $\phi$ solves the heat equation, we have
$$
\begin{aligned}
\mathcal{P}z&=\eps k (t+1)^{k-1} \phi-\eps^p (t+1)^{kp} \phi^p-b\eps^q (t+1)^{kq} \Bigl|{|x|\over 2(t+1)}\Bigr|^q \phi^q \cr
&=\eps (t+1)^{k-1}\phi\Bigl\{k-\eps^{p-1} (t+1)^{k(p-1)+1} \phi^{p-1}-b\eps^{q-1} (t+1)^{k(q-1)+1} 
\Bigl|{|x|\over 2(t+1)}\Bigr|^q\phi^{q-1}\Bigr\}.
\end{aligned}
$$
Using
$\phi^{p-1}\le (t+1)^{-{n(p-1)\over 2}}$
and
$$
b\Bigl|{|x|\over 2(t+1)}\Bigr|^q\phi^{q-1}
=(t+1)^{-{n(q-1)+q\over 2}} {|x|^q\over 2^q(t+1)^{q\over 2}}e^{-{(q-1)|x|^2\over 4(t+1)}}\le C(t+1)^{-{n(q-1)+q\over 2}},
$$
it follows that
$$
\begin{aligned}
\eps^{-1}\phi^{-1} (t+1)^{1-k}\
&\ge k-\eps^{p-1} (t+1)^{k(p-1)+1-{n(p-1)\over 2}}-C\eps^{q-1} (t+1)^{k(q-1)+1-{n(q-1)+q\over 2}}.
\end{aligned}
$$
Now, owing to \eqref{hyp2}, we have
${n(p-1)\over 2}>1$ and ${n(q-1)+q\over 2}>1$.
Therefore, we may choose $0<k<{n\over 2}$ small enough so that
$$
k(p-1)+1-{n(p-1)\over 2}\le 0\quad\hbox{and}\quad k(q-1)+1-{n(q-1)+q\over 2}\le 0.
$$
Then taking $\eps>0$ sufficiently small, it follows that 
$\mathcal{P}z\ge 0$ in $(0,\infty)\times \R^n$.
If $0\le u_0(x)\le \eps e^{-\frac{|x|^2}{4}}$, it then follows from the comparison principle (see e.g. \cite[Section 52]{QS07}) that
$$
%0\le 
u(t,x)\le z(t,x)\le \eps(t+1)^{k-\frac{n}{2}}\le \eps,\quad (t,x)\in (0,T)\times  \R^n.
$$
Finally, by Proposition \ref{PR1}, one deduces that $T=\infty$.  This completes the proof of Theorem \ref{T1}.
\end{proof}

\begin{rem}
 For $p,q,u_0$ as in Theorem \ref{T1}(ii), with a little more work (at the level of the proof of Proposition \ref{PR1}), 
one can show that $|\nabla u|$ is actually bounded.
\end{rem}

Now, we study problem \eqref{pbm1}  under the assumption 
$\int_{\R^n}u_0(x)\,dx \ge 0$. In particular, we focus our attention on sign-changing solutions. Our second main result is given by the following theorem. 

\begin{thm}\label{T2}
Let $n\geq 1$, $p,q>1$ and $b>0$.  Assume 
\be\label{blcd}
 p\le p_F\quad\hbox{or}\quad (q-1)(np-1) \le 1.
\ee
Then  for any $u_0\in BC^1(\R^n)  \cap L^1(\R^n)$ such that $\int_{\R^n}u_0(x)\,dx \ge 0$  and $u_0\not\equiv 0$, the solution to problem \eqref{pbm1} blows up in finite time.
\end{thm}

\begin{rem}
 (i) Assume $p>p_F$. We note that the assumption $(q-1)(np-1)\le 1$ amounts to 
$$q\le q_1(n,p):=1+\frac{1}{np-1}<1+\frac{1}{n+1},$$
so that the assumption on $q$ in Theorem~\ref{T2} is more restrictive than in the case of positive solutions in Theorem~\ref{T1}(i).
It is an open question 
whether there exist global-sign changing solutions to problem \eqref{pbm1} with $\int_{\R^n} u_0(x)\,dx\ge 0$ 
 for $q\in \bigl(1+\frac{1}{np-1},1+\frac{1}{n+1}\bigr]$.
\smallskip

 (ii) The assumption $u_0\in L^1(\R^n)$ in Theorem~\ref{T2} may be replaced with $(u_0)_-\in L^1(\R^n)$,
and the conclusion then remains true when $\int_{\R^n}u_0(x)\,dx=+\infty$.
\end{rem}

\begin{proof}[Proof of Theorem \ref{T2}.]
 When $p\le p_F$, the result follows from \cite[Section 26]{MiPo01}
(which applies to supersolutions of problem \eqref{pfuj} under our assumption on $u_0$).
\smallskip

Let us thus assume $(q-1)(np-1)\le 1$.
The proof is based on a  rescaled test-function argument (see \cite{MiPo01} for a general account of these methods).
Let us suppose that $u$ is a global solution to \eqref{pbm1}. 
Given $\tau>0$, we introduce the test function $\psi=\psi_\tau$ given by
$$
\psi(t,x)=f(t) g(x),\quad (t,x)\in (0,\tau)\times \mathbb{R}^n,
$$
with
\be \label{fg}
 f(t)=\theta^{\frac{p}{p-1}}\left(\frac{t}{\tau}\right)
\quad\mbox{and}\quad g(x)=\xi^{\frac{q}{q-1}}\left(\frac{x}{\tau^{r}}\right),
\ee
 where $\theta\in C^\infty(\R_+)$, $\xi\in C^\infty(\R^n)$ satisfy
$$
0\leq \theta, \xi\leq 1,
\qquad
\theta(\sigma)=\begin{cases}
1 &\text{if }0\leq \sigma\leq 1,\\
0 &\text{if } \sigma \geq 2,
\end{cases}
\qquad
\xi(z)=\begin{cases}
1 &\text{if }|z|\leq 1,\\
0 &\text{if } |z| \geq 2,
\end{cases}
$$  
and $r>0$ is to be chosen. Next, multiplying \eqref{pbm1} by $\psi$ and integrating over $Q_\tau:=(0, 2\tau)\times \R^n$, we obtain
$$
\int_{Q_{\tau}} u_t \psi \,dx\,dt -\int_{Q_{\tau}}  \psi \Delta u \,dx\,dt
=\int_{Q_{\tau}} \left(|u|^p+b|\nabla u|^q\right)\psi\,dx\,dt.
$$
Using an integration by parts, there holds
$$
-\int_{\R^n}u_0(x) g(x)\,dx -\int_{Q_\tau} u\psi_t\,dx\,dt +\int_{Q_\tau}\nabla u\cdot\nabla \psi\,dx\,dt=\int_{Q_{\tau}} \left(|u|^p+b|\nabla u|^q\right)\psi\,dx\,dt,
$$ 
which yields
\be\label{Fest}
\int_{Q_{\tau}} \left(|u|^p+b|\nabla u|^q\right)\psi\,dx\,dt+
\int_{\R^n}u_0(x) g(x)\,dx  \leq \int_{Q_\tau} |u| |\psi_t|\,dx\,dt+
\int_{Q_\tau}|\nabla u||\nabla \psi|\,dx\,dt.
\ee
Further, by H\"older's inequality and the support property of $\theta$, we obtain
$$
\begin{aligned}
\int_{Q_\tau} |u| |\psi_t|\,dx\,dt
&=\frac{p\tau^{-1}}{p-1}\int_{Q_\tau} |u|\theta^{\frac{1}{p-1}}(\tau^{-1}t)g^{\frac{1}{p}}(x)\, |\theta_t(\tau^{-1}t)|g^{\frac{p-1}{p}}(x)\,dx\,dt \\
&\le\frac{p\tau^{-1}}{p-1}\left(\int_\tau^{2\tau}\int_{\R^n} |u|^p\psi \,dx\,dt\right)^{\frac{1}{p}}
\left(\int_\tau^{2\tau}\int_{\R^n} |\theta_t(\tau^{-1}t)|^{\frac{p}{p-1}}g(x)\,dx\,dt\right)^{\frac{p-1}{p}}, \\
\end{aligned}
$$
hence
\be\label{Holder1}
\int_{Q_\tau} |u| |\psi_t|\,dx\,dt
\le  C I_1^{\frac{p-1}{p}}\left(\int_\tau^{2\tau}\int_{\R^n} |u|^p\psi \,dx\,dt\right)^{\frac{1}{p}},
\ee
with
$$
I_1(\tau):= \tau^{-\frac{p}{p-1}}\int_{Q_\tau} |\theta_t(\tau^{-1}t)|^{\frac{p}{p-1}}g(x)\,dx\,dt.
$$
Here and below, $C$ is a  positive  constant (independent of $\tau$),  whose value may change from line to line.  
 In particular, using $\varepsilon$-Young inequality (with $\varepsilon =\frac{1}{2}$), we have
\be\label{esI}
\int_{Q_\tau} |u| |\psi_t|\,dx\,dt\leq \frac{1}{2} \int_{Q_{\tau}} |u|^p\psi\,dx\,dt
 +CI_1(\tau).
\ee
Similarly, we get
$$
\begin{aligned}
\int_{Q_\tau} |\nabla u| & |\nabla \psi|\,dx\,dt
=\frac{q\tau^{-r}}{q-1}\int_{Q_\tau} |\nabla u| f^{\frac{1}{q}}(t)\xi^{\frac{1}{q-1}}(\tau^{-r}x)\, f^{\frac{q-1}{q}}(t)|\nabla\xi(\tau^{-r}x)|\,dx\,dt \\
&\le\frac{q\tau^{-r}}{q-1}\left(\int_0^{2\tau}\int_{\tau^r\le|x|\le 2\tau^r} |\nabla u|^q\psi \,dx\,dt\right)^{\frac{1}{q}} 
 \left(\int_{Q_\tau} |\nabla\xi(\tau^{-r}x)|^{\frac{q}{q-1}}f(t)\,dx\,dt\right)^{\frac{q-1}{q}}, \\
\end{aligned}
$$
hence
\be\label{Holder2}
\int_{Q_\tau} |\nabla u| |\nabla \psi|\,dx\,dt
\le  C I_2^{\frac{q-1}{q}}\left(\int_0^{2\tau}\int_{\tau^r\le|x|\le 2\tau^r} |\nabla u|^q\psi \,dx\,dt\right)^{\frac{1}{q}}
\ee
with
$$
I_2(\tau):= \tau^{-\frac{rq}{q-1}} \int_{Q_\tau} |\nabla\xi(\tau^{-r}x)|^{\frac{q}{q-1}}f(t)\,dx\,dt.
$$
Using $\varepsilon$-Young inequality (with $\varepsilon =\frac{b}{2}$), we obtain in particular
\be\label{esII}
\int_{Q_\tau} |\nabla u| |\nabla \psi|\,dx\,dt\le \frac{b}{2} 
\int_{Q_{\tau}} |\nabla u|^q\psi\,dx\,dt +CI_2(\tau).
\ee
By \eqref{Fest}, \eqref{esI} and \eqref{esII}, it folows that
\be\label{esIII}
\frac{1}{2}\int_{Q_{\tau}} \left(|u|^p+b|\nabla u|^q\right)\psi\,dx\,dt+
\int_{\R^n}u_0(x) g(x)\,dx  \leq C\left(I_1(\tau)+I_2(\tau)\right),
\ee

Let us next estimate $I_i(\tau)$, $i=1,2$.
 By the properties of $\psi, g$ and the change of variables $t=\tau s$, $x=\tau^r y$, we deduce that 
$$
\begin{aligned}
I_1(\tau)
&= \tau^{-\frac{p}{p-1}}\left(\int_0^{2\tau} |\theta_t(\tau^{-1}t)|^{\frac{p}{p-1}}\,dt\right)
\left(\int_{|x|\le 2\tau^r} \xi^{\frac{q}{q-1}}(\tau^{-r}x)\,dx\right) \\
&= \tau^{nr+1-\frac{p}{p-1}}\left(\int_0^2 |\theta_t(s)|^{\frac{p}{p-1}}\,ds\right)
\left(\int_{|y|\le 2} \xi^{\frac{q}{q-1}}(y)\,dy\right),
\end{aligned}
$$
hence
\be \label{estI1tau}
I_1(\tau)=C \tau^{nr-\frac{1}{p-1}}.
\ee
 Similarly, for $I_2(\tau)$, we compute
$$
\begin{aligned}
I_2(\tau)
&=\tau^{-\frac{rq}{q-1}}
\left(\int_0^{2\tau}\theta^{\frac{p}{p-1}}(\tau^{-1}t)\,dt \right) \left(\int_{|x|\le 2\tau^r} |\nabla\xi(\tau^{-r}x)|^{\frac{q}{q-1}}\, dx\right) \\
&= \tau^{nr+1-\frac{rq}{q-1}}\left(\int_0^2 \theta^{\frac{p}{p-1}}(s)\,ds\right)
\left(\int_{|y|\le 2} |\nabla\xi(y)|^{\frac{q}{q-1}}\,dy\right),
\end{aligned}
$$
hence 
\be \label{estI2tau}
I_2(\tau)=C \tau^{1+nr-\frac{rq}{q-1}}.
\ee

\medskip

Now,  combining \eqref{esIII}--\eqref{estI2tau}, we obtain
\be \label{TEST}
\begin{aligned}
\frac{1}{2}\int_{Q_{\tau}} 
&\left(|u|^p+b|\nabla u|^q\right)\psi\,dx\,dt+ \int_{\R^n} u_0(x) \xi^{\frac{q}{q-1}}(\tau^{-r}x)\,dx 
=\frac{1}{2}\int_{Q_{\tau}} \left(|u|^p+b|\nabla u|^q\right)\psi\,dx\,dt\\
&+\int_{\R^n}u_0(x) g(x)\,dx 
\le C \left(\tau^{nr-\frac{1}{p-1}}+\tau^{nr+1-\frac{rq}{q-1}}\right),\quad \tau>0. \\
\end{aligned}
\ee
Let us make the choice
$r=\frac{p(q-1)}{q(p-1)}$. With such $r$, one has
\be \label{propchoicer}
nr-\frac{1}{p-1}=1+nr-\frac{rq}{q-1}=\frac{(q-1)(np-1)-1}{q(p-1)}
\ee
and we deduce from \eqref{TEST} that, for all $\tau>0$,
\be \label{estcombined}
\frac{1}{2}\int_{Q_{\tau}} \left(|u|^p+b|\nabla u|^q\right)
\theta^{\frac{p}{p-1}}\Bigl(\frac{t}{\tau}\Bigr)\xi^{\frac{q}{q-1}}\Bigl(\frac{x}{\tau^{r}}\Bigr)\,dx\,dt 
+\int_{\R^n} u_0(x) \xi^{\frac{q}{q-1}}\Bigl(\frac{x}{\tau^{r}}\Bigr)\,dx\leq C \tau^\frac{(q-1)(np-1)-1}{q(p-1)}.
\ee

 First consider the case $(q-1)(np-1)<1$.
Using Fatou's lemma in the first integral and the dominated convergence theorem in the second one, 
recalling $u_0\in L^1(\R^n)$, we may
pass to the infimum limit as $\tau\to \infty$ in inequality \eqref{estcombined} to obtain
$$
\frac{1}{2}\int_0^\infty\int_{\R^n} \left(|u|^p+b|\nabla u|^q\right)\,dx\,dt 
\le \frac{1}{2}\int_0^\infty\int_{\R^n} \left(|u|^p+b|\nabla u|^q\right)\,dx\,dt +  \int_{\R^n} u_0(x)\,dx\leq 0,
$$
 hence $u\equiv 0$.

Finally let us consider the case $(q-1)(np-1)=1$. 
Then  \eqref{estI1tau}, \eqref{estI2tau},  \eqref{propchoicer} and \eqref{estcombined} guarantee that
$I_1(\tau)+I_2(\tau)\le C$ for all $\tau>0$
and that
\be \label{integrabilityprop}
\int_0^\infty\int_{\R^n} \left(|u|^p+b|\nabla u|^q\right)\,dx\,dt <\infty.
\ee
Going back to \eqref{Fest}, \eqref{Holder1}, \eqref{Holder2}, we thus deduce that
\be \label{integrabilityprop2}
\begin{aligned}
\int_{Q_{\tau}}& \left(|u|^p+b|\nabla u|^q\right)\psi\,dx\,dt
+\int_{\R^n}u_0(x) g(x)\,dx \\
&\le C\left(\int_\tau^{2\tau}\int_{\R^n} |u|^p \,dx\,dt\right)^{\frac{1}{p}}
+C\left(\int_0^{2\tau}\int_{\tau^r\le|x|\le 2\tau^r} |\nabla u|^q \,dx\,dt\right)^{\frac{1}{q}}.
\end{aligned}
\ee
But \eqref{integrabilityprop} along with the dominated convergence theorem imply that the RHS of 
\eqref{integrabilityprop2} goes to $0$ as $\tau\to\infty$.
It then follows again that $u\equiv 0$.
This completes the proof of Theorem~\ref{T2}.
\end{proof}

\section{The study of problem \eqref{pbm1b}} %of  of 
\label{sec3}

We now consider the inhomogeneous problem \eqref{pbm1b}.
Our  third main result is given by the following theorem. 
Here and in what follows we use the convention $n/(n-2)=\infty$ for $n=2$.

\medskip

\begin{thm}\label{T3}
Let $n\geq 2$, $p,q>1$ and $b>0$. 
\begin{itemize}
\item[(i)] Assume 
$$
p<\frac{n}{n-2}\quad\hbox{or}\quad q< \frac{n}{n-1}.
$$
Then  for any $u_0\in BC^1(\R^n)$ and  $h\in  BC^1(\R^n)\cap L^1(\R^n)$ such that $\int_{\R^n}h(x)\,dx >0$, the solution to problem \eqref{pbm1b} blows up in finite time.
\item[(ii)] Let $n\ge 3$ and assume that 
\be\label{hyppq2}
p>\frac{n}{n-2}\quad\hbox{and}\quad q>\frac{n}{n-1}.
\ee
Then there exists $h\in BC^1(\R^n)$ with $h>0$ in $\R^n$, such that 
problem \eqref{pbm1b} admits positive global classical solutions for suitably small positive initial data $u_0\in BC^1(\R^n)$.
\end{itemize}
\end{thm}

\begin{rem}
 (i) Observe that in  case (i) of Theorem \ref{T3} the solution to problem \eqref{pbm1b} may change sign.
  Also, the assumption $h\in L^1(\R^n)$ may be replaced with $h_-\in L^1(\R^n)$,
and the conclusion then remains true when $\int_{\R^n}h(x)\,dx=+\infty$.

 \smallskip

(ii) Similarly to problem \eqref{pbm1}, a phenomenon of discontinuity of the
Fujita exponent in the presence of the gradient term occurs.  Namely, as long as $q>\frac{n}{n-1}$ , the Fujita exponent remains the same as for the problem without gradient term, i.e. $p=\frac{n}{n-2}$, and when $q$ becomes smaller than $\frac{n}{n-1}$ ($n\geq 3$), the Fujita exponent jumps to $p=\infty$. 
\end{rem}

\begin{proof}[Proof of Theorem \ref{T3}.]
(i) When $p<\frac{n}{n-2}$, the result follows immediately from \cite[Theorem 2.1]{BLZ}. Let us thus assume $q< \frac{n}{n-1}$. Let us suppose that $u$ is a global solution to \eqref{pbm1b}. Given $\tau > 0$, we introduce the test function $\varphi=\varphi_\tau$ given by
$$
\varphi(t,x)=a(t) g(x),\quad (t,x)\in (0,\tau)\times \mathbb{R}^n,
$$
with $g$ is given by \eqref{fg} and 
$$
a(t)=\eta^{\frac{p}{p-1}}\left(\frac{t}{\tau}\right),
$$
where $\eta\in C^\infty(\R)$ is such that 
$$
\eta\geq 0,\quad \eta\not\equiv 0\quad \mbox{and}\quad \mbox{supp}(\eta) \subset (0,1).
$$
Next, multiplying \eqref{pbm1b} by $\varphi$,  integrating over $S_\tau:=(0, \tau)\times \R^n$, and using an integration by parts,  we obtain
$$
-\int_{S_\tau} u\varphi_t\,dx\,dt +\int_{S_\tau}\nabla u\cdot\nabla \varphi\,dx\,dt=\int_{S_{\tau}} \left(|u|^p+b|\nabla u|^q+h(x)\right)\varphi\,dx\,dt,
$$ 
which yields
$$
\int_{S_{\tau}} \left(|u|^p+b|\nabla u|^q\right)\varphi\,dx\,dt+
\int_{S_\tau}h(x) \varphi\,dx\,dt  \leq \int_{S_\tau} |u| |\varphi_t|\,dx\,dt+
\int_{S_\tau}|\nabla u||\nabla \varphi|\,dx\,dt.
$$
Further, following a similar argument as in the proof of Theorem \ref{T2}, one obtains
\be \label{esh1}
\frac{1}{2}\int_{S_{\tau}} 
\left(|u|^p+b|\nabla u|^q\right)\varphi\,dx\,dt+ \int_{S_\tau} h(x)\varphi\,dx\,dt 
\le C \left(\tau^{nr-\frac{1}{p-1}}+\tau^{nr+1-\frac{rq}{q-1}}\right),\quad \tau>0.
\ee
On the other hand, one has
\be\label{esth2}
\int_{S_\tau} h(x)\varphi\,dx\,dt =\left(\int_0^\tau a(t)\,dt \right)
\left(\int_{\R^n} h(x) g(x)\,dx\right)
= C\tau \int_{\R^n} h(x) \xi^{\frac{q}{q-1}}\left(\frac{x}{\tau^{r}}\right)\,dx.
\ee
Hence, combining \eqref{esh1} with \eqref{esth2}, there holds
\be\label{esth3}
\int_{\R^n} h(x) \xi^{\frac{q}{q-1}}\left(\frac{x}{\tau^{r}}\right)\,dx\le C\left(\tau^{nr-\frac{1}{p-1}-1}+\tau^{nr-\frac{rq}{q-1}}\right),\quad \tau>0.
\ee
Now, let us make the choice $r=\frac{p(q-1)}{q(p-1)}$. With such $r$, one has
$$
nr-\frac{1}{p-1}-1=nr-\frac{rq}{q-1}=\frac{p}{p-1}  \left(\frac{n(q-1)-q}{q}\right)
$$
and we deduce from \eqref{esth3} that, for all $\tau>0$,  
\be\label{esth4}
\int_{\R^n} h(x) \xi^{\frac{q}{q-1}}\left(\frac{x}{\tau^{r}}\right)\,dx\le C \tau^{\frac{p}{p-1}  \left(\frac{n(q-1)-q}{q}\right)}.
\ee
Since $q<\frac{n}{n-1}$ and $h\in L^1(\R^n)$, passsing to the limit as $\tau\to \infty$ in \eqref{esth4} and using the dominated convergence theorem, 
we obtain
$$
\int_{\R^n} h(x)\,dx\leq 0,
$$
which contradicts the fact that $\int_{\R^n} h(x)\,dx>0$. 

\smallskip 

(ii) We look for a supersolution of the form $v(x)=\eps(1+|x|^2)^{-k}$, where $k,\eps>0$, as well as $h$,
will determined below. We compute
$$
|\nabla v(x)|=2\eps k|x|(1+|x|^2)^{-k-1}\le 2\eps k(1+|x|^2)^{-k-{1\over 2}},$$
$$
\Delta v(x)=2\eps k\Bigl[-n(1+|x|^2)^{-k-1}+2(k+1)|x|^2(1+|x|^2)^{-k-2}\Bigr]
\le 2\eps k{2 (k+1)-n\over (1+|x|^2)^{k+1}}.
$$
Therefore,
$$
\Delta v+v^p+b|\nabla v|^q
\le \eps(1+|x|^2)^{-k-1}\Bigl[2k\bigl(2(k+1)-n\bigr)+\frac{\eps^{p-1}}{(1+|x|^2)^{k(p-1)-1}}
+\frac{2^qb\eps^{q-1}}{(1+|x|^2)^{k(q-1)+\frac{q}{2}-1}}\Bigr].
$$
Owing to assumption \eqref{hyppq2}, we may choose $k>0$ such that
$$
\textstyle\max\Bigl\{\frac{1}{p-1},\frac{2-q}{2(q-1)}\Bigr\}<k<\frac{n-2}{2}.
$$
Then choosing $\eps>0$ sufficiently small, we get
$$
\Delta v+v^p+b|\nabla v|^q
\le \eps(1+|x|^2)^{-k-1}\Bigl[2k\bigl(2(k+1)-n\bigr)+\eps^{p-1}+2^qb\eps^{q-1}\Bigr]<0
\quad\hbox{ in $\R^n$.}
$$
Setting $h(x):=-\Delta v-v^p-b|\nabla v|^q>0$, we see that $v$ is a (super-)solution to the PDE in \eqref{pbm1b}. Starting from any initial data $u_0\in BC^1(\R^n)$ with $0\le u_0\le v$, we deduce from the comparison principle
that $u(t,x)\le v(x)$ in $(0,T)\times  \R^n$. Applying Proposition \ref{PR1}, we conclude that $T=\infty$.
\end{proof}

\noindent{\bf Acknowledgements.} The first and second authors extend their appreciation to the Deanship of Scientific Research at King Saud University for funding this work through research group No. RGP--237.

\vskip 3mm

\font\tenrm=cmr10
{\tenrm
\noindent Mohamed Jleli,
Department of Mathematics, College of Science, King Saud University, P.O. Box 2455, Riyadh, 11451, Saudi Arabia. 
E-mail: jleli@ksu.edu.sa
\vskip 1.5mm

\noindent Bessem Samet,
Department of Mathematics, College of Science, King Saud University, P.O. Box 2455, Riyadh, 11451, Saudi Arabia. 
E-mail: bsamet@ksu.edu.sa
\vskip 1.5mm

\noindent Philippe Souplet,
 Universit\'e Paris 13, Sorbonne Paris Cit\'e,
CNRS UMR 7539, Laboratoire Analyse, G\'{e}om\'{e}trie et Applications,
93430 Villetaneuse, France. 
E-mail: souplet@math.univ-paris13.fr 

}

\end{document}